\documentclass[11pt,twoside]{amsart}
\usepackage[english]{babel}
\usepackage{amsmath, amsthm, amssymb, amsfonts, cite}
\usepackage{latexsym}
\usepackage{graphicx}
\usepackage {bm}
\usepackage {indentfirst} 
\usepackage{srcltx}
\usepackage{microtype}

\usepackage[pagebackref,colorlinks=true,linkcolor=blue,citecolor=blue,urlcolor=red, hypertexnames=true]{hyperref}

\advance\hoffset by -1.5cm

\newcommand{\B}{\mathcal{B(H)}}
\newcommand{\R}{\mathcal{R}}
\newcommand{\n}{\mathcal{N}}
\newcommand{\h}{\mathcal{H}}
\newcommand{\m}{\mathcal{M}}
\newcommand{\ka}{\mathcal{K}}

\newcommand{\prep}{\perp}

\newtheorem{example}{Example}[section]
\newtheorem{theorem}[example]{Theorem}

\newtheorem{proposition}[example]{Proposition}
\newtheorem{corollary}[example]{Corollary}
\newtheorem{remark}[example]{Remark}

\numberwithin{equation}{section}

\title[{Left invertible quasi-isometric liftings}]{Left invertible quasi-isometric liftings}

\author[L. Suciu]{Laurian Suciu} 
\author[A.-M. Stoica]{Andra-Maria Stoica}
\address{Department of Mathematics and Informatics, ``Lucian Blaga'' University of Sibiu, Romania}
\email{laurians2002@yahoo.com, andra\textunderscore maria44@yahoo.com}

\begin{document}

	\begin{abstract}
	Quasi-isometric liftings similar to isometries, for the operators similar to contractions in Hilbert spaces, are investigated. 
The existence of such liftings is established, and their applications are explored for specific operator classes, including quasicontractions. A particular focus is placed on operators that admit left invertible minimal quasi-isometric liftings. These operators are characterized within the framework of $A$-contractions, and the matrix structures of their liftings are analyzed, highlighting parallels with the isometric liftings of contractions.

	\end{abstract}
	\maketitle
	\section{Introduction and preliminaries}
	\medskip
	\subsection*{Introduction}
	
		It is well known that the contractive operators on a (complex) Hilbert space $\h$, are characterized by the fact that they have isometric liftings acting on larger Hilbert spaces $\ka$ than $\h$. This is the basic result in the Sz.-Nagy-Foias dilation theory of contractions, which has many applications in operator theory, and not only (see \cite{FF,SFb}). 
		
		The concept of lifting was recently studied in the context of $m$-isometries ($m\ge 2$ an integer), see \cite{BS2,BMS,MajSu-lift,S-2020,S_Monast}. The class of $m$-isometric operators was initially developed by J. Agler and M. Stankus \cite{AS1,AS2,AS3}, with further contributions by A. Aleman \cite{Aleman}, S. McCullough \cite{McC}, S. Richter \cite{Richter0}, and many others. Within this framework, it has been shown in \cite{BS2,MajSu-lift,S-2020,S_Monast} that every operator $T$ similar to a contraction on $\h$ admits a 2-isometric lifting $S$ on $\ka \supset \h$, this meaning that $S^{*2}S^2-2S^*S+I=0$. But such operator $T$ is power bounded, i.e., $\sup_{n\ge 1} \|T^n\|<\infty$, while $\sup_{n\ge 1} \frac{1}{\sqrt{n}}\|S^n\|<\infty $. 
		Consequently, 2-isometric liftings, although theoretically significant, may not represent the most natural or closest liftings for operators similar to contractions.
A more compelling and natural class of liftings for such operators is given by quasi-isometries, these being operators that are isometric on their ranges. The quasi-isometric liftings better capture the structure of operators similar to contractions, offering an improved framework for their analysis.
		
		
		 
		 
A study of quasi-isometric liftings was initiated in \cite{SS1}, where it was shown that every operator $T$ on $\h$ similar to a contraction, admits a quasi-isometric lifting $S$ on $\ka \supset \h$ satisfying $S^*S\h \subset \h$. Furthermore, \cite{SS1} analyzed various special cases of operators $T$, examining additional properties of $S$, including minimal liftings and those corresponding to quasicontractions or operators similar to isometries.

In this paper, we continue the study of quasi-isometric liftings, focusing on those that are left invertible. Specifically, we demonstrate that every operator $T$ similar to a contraction admits a left invertible quasi-isometric lifting $S$, but $\h$ is not necessarily invariant for $S^*S$. Notably, we show that this is even true for operators similar to contractive symmetries.

Furthermore, we characterize operators $T$ with liftings $S$ as described above, which satisfy the condition $S^*S\h\subset \h$. The matrix structure of these liftings closely resembles that of isometric liftings for contractions, making them valuable in applications. Beyond general results, we also address more specific classes of operators with left invertible quasi-isometric liftings, that satisfy the aforementioned invariance condition.

		 \subsection*{Definitions and notation}
		 For every complex Hilbert spaces $\h$ and $\ka$ we denote by $\mathcal{B}(\h,\ka)$ the Banach space of all bounded operators from $\h$ into $\ka$, while $\B=\mathcal{B}(\h,\h)$ is regarded as a $C^*$-algebra with the unit element $I(=I_{\h}$, when is nedeed). For $T\in \mathcal{B}(\h,\ka)$, $\R(T)\subset \ka$ and $\n(T) \subset \h$ are the range and the kernel of $T$, respectively, while $T^*\in \mathcal{B}(\ka,\h)$ is the adjoint operator of $T$. The operator $T$ is left invertible if it is injective with closed range, that is $T^*T$ is invertible in $\B$.
		 
		 For a positive operator $T\in \B$ we write $T^{1/2}$ for the square root of $T$. This means that $(Th,h)\ge 0$ for any $h\in \h$ and $(T^{1/2})^2=T$, where $(\cdot,\cdot)$ stands for the scalar product in $\h$. In this case we write $T\ge 0$, and for $T,T'\in \B$ the notation $T\le T'$ means that $T'-T\ge 0$.
		 
		 An operator $T\in \B$ is contractive (expansive) if $T^*T\le I$ (respectively, $I\le T^*T$). Also, $T$ is an isometry if $T^*T=I$, and a unitary operator if $T$ and $T^*$ are isometries. 
		 
		 For a Hilbert space $\mathcal{E}$, we consider the Hilbert space $\ell_+^2(\mathcal{E})=\bigoplus_{n=0}^{\infty}\mathcal{E}_n$, where $\mathcal{E}_n=\mathcal{E}$ for $n\ge 0$, and $\mathcal{E}_0=\mathcal{E}\oplus \{0\}=\n(S_+^*)$, $S_+$ being the forward shift on $\ell_+^2(\mathcal{E})$, so an isometry. 
		 
		 A closed subspace $\m$ of $\h$ is invariant (reducing) for $T\in \B$ if $T\m\subset \m$ (respectively, $T\m\subset \m$ and $T^*\m\subset \m$). In this case, $T|_{\m}$ stands for the restriction of $T$ to $\m$. Also, by $J_{\m,\h}\in \mathcal{B}(\m,\h)$ we understand the natural injection of $\m$ into $\h$, $P_{\h,\m}=J_{\m,\h}^*$ is the projection of $\h$ onto $\m$, while $P_{\m}\in \B$ is the orthogonal projection with $\R(P_{\m})=\m$. 
		 
		 When $\ka$ contains $\h$ as a closed subspace, an operator $S\in \mathcal{B}(\ka)$ is a {\it lifting} for $T\in \B$ if $P_{\ka,\h}S=TP_{\ka,\h}$. In this case, $S(\ka \ominus \h)\subset \ka \ominus \h$ and $S^*J_{\h,\ka}=J_{\h,\ka}T^*$, that is $S^*$ is an {\it extension} for $T^*$.
		 
In this paper we refer to operators $T\in \B$ which are similar to contractions, this meaning that there exists an invertible and positive operator $A\in \B$, such that $T^*AT\le A$. Clearly, $A$ can be chosen also satisfying the condition $A\ge T^*T$. Between these operators we distinguish the quasicontractions $T$, meaning $T^{*2}T^2\le T^*T$, and particularly the quasi-isometries, that is those with $T^{*2}T^2=T^*T$.

According to \cite[Theorem 2.1]{SS1}, every operator $T$ similar to a contraction admits a quasi-isometric lifting $S$ on a space $\ka \supset \h$ such that $S^*S\h\subset \h$. In addition, $S$ can be chosen satisfying the minimality condition $\ka=\bigvee_{n\ge 0} S^n\h$ (i.e., $\ka$ is the closed linear span of the set $\{S^n\h:n\ge 0\}$. Such a lifting $S$ with these two  conditions (of invariance and minimality) will be called a {\it natural quasi-isometric lifting} for $T$. 
 		
\subsection*{Organization of the paper}



This paper focuses on left invertible quasi-isometric liftings.

In Section 2, we begin by characterizing the operators 
$T$ that admit left invertible and natural quasi-isometric liftings. We derive a range condition involving $T$ and an intertwiner $A$ of $T$ with a contraction. Additionally, we present a matrix structure for such liftings $S$, which reveals further properties of $S$, beyond its role as a natural lifting for $T$. Moreover, we describe quasicontractions in terms of left invertible quasi-isometric liftings $S$ that satisfy the condition $\ka \ominus \h\subset \n(S^*S-I)$. For these cases, the matrix structure of $S$ is simpler and closely resembles that of isometric liftings.

In Section 3, we prove that every operator $T$ similar to a contraction admits a left invertible and minimal quasi-isometric lifting $S$, which does not necessarily satisfy the condition $S^*S\h \subset \h$. We provide a class of operators, including a concrete example, to illustrate this phenomenon. Additionally, we revisit the range condition from Section 2, offering further insights and results. Finally, we identify another class of operators that admit left invertible and natural quasi-isometric liftings.

	\medskip
	
	\section{Left invertible and natural quasi-isometric liftings}
\medskip
	


	It is well known that every operator similar to a contraction admits a lifting that is similar to an isometry, and thus a left invertible lifting. However, such liftings are not generally quasi-isometric.
Our first results focus on natural quasi-isometric liftings.

	\begin{theorem} \label{thm21}
		An operator $T \in \mathcal{B}(\h)$ has a left invertible quasi-isometric lifting $S$ on a space $\ka \supset \h$ with $S^*S\h \subset \h$, if and only if there exists an invertible operator $A \in \mathcal{B}(\h)$ such that
		\begin{equation}
			A \geq T^*T\text{, } \quad A \geq T^*AT \text{, } \quad \mathcal{R}[(A-T^*T)^{1 / 2}]=\mathcal{R}[(A-T^*AT)^{1 / 2}]. \label{eq21}
		\end{equation}
		
		In this case, $S$ can be chosen to be also a minimal lifting for $T$. 
	\end{theorem}
	
	\begin{proof}
		Let $T$ and $S$ be as above, $S$ having on $\ka=\h^{\prep} \oplus \h$ the block representation
		\begin{equation}
			S=\begin{pmatrix}
				W & E \\
				0 & T
			\end{pmatrix}\text{, where } W=S|_{\h^\prep}\text{, }\quad E=P_{\h^{\perp}}S|_\h. \label{eq22}
		\end{equation}
		
		As $S^*S\h \subset \h$ we have $W^*E=0$, and so we get
		\begin{equation}
			S^*S=\begin{pmatrix}
				W^*W & 0 \\
				0 & A
			\end{pmatrix} \text{, where } A=E^*E+T^*T. \label{eq23}
		\end{equation}
		
		On the other hand, $S$ and $W$ are left invertible quasi-isometries, so $A$ is invertible and $A \geq T^*T$. Also, $W$ and $W^*W$ have on $\h ^\prep=\mathcal{R}(W) \oplus \n (W^*)$ the representations
		\begin{equation}
			W=\begin{pmatrix}
				V & X \\
				0 & 0
			\end{pmatrix}\text{, } \quad W^*W=\begin{pmatrix}
			I & V^*X \\
			X^*V & X^*X
			\end{pmatrix}, \label{eq24}
		\end{equation}
		where $V=W|_{\mathcal{R}(W)}$ is an isometry and $X=P_{\mathcal{R}(W)}W|_{\n(W^*)}$. 
		Since $W^*W \geq c I$ for some constant $c>0$, it follows that $X^*X \geq c I_{\h}$, meaning that $X$ is left invertible from $\n(W^*)$ into $\mathcal{R}(W)$.
		
		Now, since $\mathcal{R}(E) \subset \n(W^*)$ we infer that $E:\h \to \mathcal{R}(W) \oplus \n(W^*)$ takes the form $E=\begin{pmatrix}
			0 & E_0
		\end{pmatrix}^{\rm tr}$ with $E_0\in \mathcal{B}(\h, \n(W^*))$. Then, using the representations \eqref{eq22} and $\eqref{eq23}$, as well as $S^{*2}S^2=S^*S$, we obtain 
		$$S^{*2}S^2=\begin{pmatrix}
			W^{*2}W^2 & W^{*2}WE \\
			E^*W^*W^2 & E^*W^*WE+T^*AT
		\end{pmatrix}=\begin{pmatrix}
		W^*W & 0 \\
		0 & A
		\end{pmatrix},$$
		hence $E^*W^*WE+T^*AT=A$. So $A-T^*AT \geq 0$, and using the form of $W$ in $\eqref{eq24}$ and that $E=\begin{pmatrix}
			0 & E_0
		\end{pmatrix}^{\rm tr}$, we get $E^*W^*WE=E^*_0X^*XE_0$. Thus, the previous relation becomes 
		$$E^*_0X^*XE_0=A-T^*AT.$$
		Next, according to the polar decomposition of $XE_0$ we obtain 
		$$XE_0=J|XE_0|=J(A-T^*AT)^{1/2}\text{, } \text{where} \quad |XE_0|=(E^*_0X^*XE_0)^{1/2}.$$
		
		Since $X$ is left invertible, there exists an operator $Y:\mathcal{R}(W) \to \n (W^*)$ such that $YX=I$, which, by the above relation, gives for $E_0$ the expression
		$$E_0=YJ(A-T^*AT)^{1/2}.$$
		But $E^*_0E_0=E^*E=A-T^*T$ from \eqref{eq23}, so by the polar decomposition of $E_0$ we have 	
		$$
		E_0=J_0(A-T^*T)^{1/2},
		$$ 
		with $J_0$ a partial isometry.
		
		Clearly, this relation shows that $E_0\neq 0$ on $\overline{\mathcal{R}(E^*_0)}=\overline{\mathcal{R}(E^*_0E_0)}=\overline{\mathcal{R}(A-T^*AT)}$, having in view that $\n(E_0)=\n(A-T^*AT)$. Since $J_1=J_0|_{\overline{\mathcal{R}(E^*_0)}}$ is unitary between $\overline{\mathcal{R}(E_0^*)}$ and $\overline{\mathcal{R}(E_0)}$, from the above relations of $E_0$, we infer that 
		$$(A-T^*T)^{1/2}=J_1^{-1}YJ(A-T^*AT)^{1/2}.$$ 
		Finally, this gives the inclusion $\mathcal{R}[(A-T^*T)^{1/2}]\subset \mathcal{R}[(A-T^*AT)^{1/2}] $.
		
		At the same time, utilizing some of the aforementioned relations one obtains that
		$$0 \leq A-T^*AT=E^*W^*WE \leq ||W||^2E^*E,$$
		so $\mathcal{R}[(A-T^*AT)^{1/2}] \subset \mathcal{R}(|E|)=\R[(A-T^*T)^{1/2}].$
		
		 Hence $\mathcal{R}[(A-T^*T)^{1/2}]=\mathcal{R}[(A-T^*AT)^{1/2}]$, which completes the required properties for $T$ and $A$ in \eqref{eq21}. This proves the direct implication of theorem.
		
		For the converse statement, we assume that there exists an invertible operator $A$ on $\h$, which, together with $T$, satisfies the conditions \eqref{eq21}. Then, by Douglas's range inclusion theorem, there exist two (bounded linear) operators $X_0$, $X_1$ on the space
		$$\h_0:=\overline{\mathcal{R}[(A-T^*T)]}=\overline{\mathcal{R}[(A-T^*AT)]},$$
		such that
		\begin{equation}
			X_0(A-T^*T)^{1/2}=(A-T^*AT)^{1/2}\text{, } \quad X_1(A-T^*AT)^{1/2}=(A-T^*T)^{1/2}. \label{eq25}
		\end{equation} 
		
		These relations show that $X_1X_0=I_{\h_0}$, hence $X_0$ is a left invertible operator in $\mathcal{B}(\h_0)$. We use $X_0$ and $A$ to construct the desired lifting for $T$.
		
		Let $\ka=\h_1 \oplus \h_0 \oplus \h$ where $\h_1=\ell_+^2(\h_0)$, and let $S\in \mathcal{B}(\ka)$ be the lifting of $T$ which, under the decomposition of $\ka$ quoted before, has the block matrix 
		\begin{equation}
			S=\begin{pmatrix}
				S_+ & X & 0 \\
				0 & 0 & (A-T^*T)^{1/2} \\
				0 & 0 & T
			\end{pmatrix}. \label{eq26}
		\end{equation}
		
		Here $X=\widetilde{J}X_0$, $\widetilde{J}$ being the embedding mapping of the $\h_0$ into $\h_1$, while $S_+$ is the forward shift on $\h_1$, so $S^*_+X=0$. Since $S^*S\h\subset \h$ we obtain that 
		$$S^*S=\begin{pmatrix}
			I & 0 & 0 \\
			0 & X^*_0X_0& 0 \\
			0 & 0 & A
		\end{pmatrix}\text{, } S^{*2}S^2=\begin{pmatrix}
		I & 0 & 0 \\
		0 & X^*_0X_0& 0 \\
		0 & 0 & (A-T^*T)^{1/2}X^*_0X_0(A-T^*T)^{1/2}+T^*AT
		\end{pmatrix}.$$
		
		From the definition of $X_0$ in \eqref{eq25}, we have 
		$$(A-T^*T)^{1/2}X^*_0X_0(A-T^*T)^{1/2}=A-T^*AT,$$
		which subsequently implies that
		 $S^*S=S^{*2}S^2$, that is $S$ in $\eqref{eq26}$ is a quasi-isometry.
		
		 Since $A$ and $X^*_0X_0$ are invertible, it follows that $S^*S$ is also invertible, meaning that $S$ is a left invertible lifting for $T$. The converse assertion of the theorem is thus proved.
		 
		 If the lifting $S$ in \eqref{eq22} or in \eqref{eq26} is not minimal, then $S_0=S|_{\ka_0}$ where $\ka_0=\displaystyle{\bigvee_{n\ge 0}}S^n\h$, is a minimal quasi-isometric lifting for $T$. Also, $S_0$ is left invertible because
		 $$S^*_0S_0=P_{\ka_0}S^*S|_{\ka_0} \geq c I_{\ka_0}$$
		 for some constant $c>0$.
			\end{proof}
			
	\begin{remark}\label{re22}
	
	\rm The essential condition on $T$ and $A$ in Theorem $\ref{thm21}$ is the equality of the corresponding ranges. However, an inclusion between these ranges can be assumed. Indeed, for every operator $T$ similar to a contraction there exists a positive operator $A_0 \in \mathcal{B}(\h)$ with $T^*A_0T \leq A_0$. Then one can multiply $A_0$ with an appropriate  positive constant, to get an operator $A \in \mathcal{B}(\h)$ satisfying the conditions $T^*T \leq T^*AT \leq A$. So $A-T^*AT \leq A-T^*T$, which gives
	 $$\mathcal{R}[(A-T^*AT)^{1/2}]\subset\mathcal{R}[(A-T^*T)^{1/2}].$$

\end{remark}

\begin{remark}\label{re23}		
		
		\rm 
		Concerning the lifting $S$ for $T$ in Theorem \ref{thm21}, we remark that the condition $S^*S\h \subset\h$ is essential, but not restrictive in this context, because one can always find liftings with this property (by \cite[Theorem 2.1]{SS1}).
		
		Outside this condition, the lifting $S$ from \eqref{eq26} has some notable features related to its restriction to $\ka \ominus \h$, as well as its behaviour regarding $\R(S)$. We explore these aspects below. 		
		\end{remark} 	
		
		\begin{theorem}\label{thm24}
Let $T\in \B$ satisfying the equivalent conditions from Theorem \ref{thm21}. Then $T$ has a left invertible quasi-isometric lifting $S$ on a space $\ka \supset \h$, such that $S^*S\h\subset \h$ and $\R(S)=\R(W)\oplus \overline{S\h}$, where $W=S|_{\ka \ominus \h}$ satisfies the range condition $W^*W\R(W)\subset \R(W)$.

Moreover, $S$ satisfies the range condition $S^*S\R(S)\subset \R(S)$ if and only if $\R(T) \subset \n(A-I)$, where $A=S^*S|_{\h}$. In this case, $S$ is a quasicontraction and $W$ is an isometry. 		
		\end{theorem}
		
		\begin{proof}
		Let $S$ be the lifting for $T$ with the representation \eqref{eq26}, where $S^*S\h\subset \h$. Preserving the notation from the previous proof, we see from the relations \eqref{eq25} that $X_0X_1=I_{\h_0}$, hence $X_0$ is an invertible operator in $\mathcal{B}(\h_0)$, which implies that 
		$$\n(S^*_+)=\widetilde{J}\h_0=\widetilde{J}\mathcal{R}(X_0)=\mathcal{R}(X),$$
		for $S_+$ and $X$ from \eqref{eq25}. This shows for the left invertible quasi-isometry $W=S|_{\h_1 \oplus \h_0}$ in \eqref{eq26} that 
		$$\mathcal{R}(W)=\mathcal{R}(S_+)\oplus\mathcal{R}(X)=\h_1\text{, }\quad \n(W^*)=\h_0.$$  	
	
	Thus $W$ satisfies the range condition $W^*W\R(W)\subset\R(W).$

Next, we infer from \eqref{eq26} that $\R(S)=\h_1\oplus \overline{S\h}=\R(W)\oplus \overline{S\h}$, where 
$$
S\h=\{J_{\h_0}(A-T^*T)^{1/2}h\oplus Th:h\in \h \} \subset \h_0 \oplus \h,
$$
$J_{\h_0}$ being the embedding mapping of $\overline{\R(A-T^*T)}$ into $\h_0=\ka \ominus (\h_1\oplus \h)$ and $A=S^*S|_{\h}$. But $S^*S|_{\h_1}=I$, which means $S(\ka \ominus \h)=\h_1=S^*S\h_1$. We want to see if $S^*S(\overline{S\h})$ can be expressed in the same way, in order to analyze the range condition for $S$.

For this, from the above structure of $\R(S)$ and the matrix of $S^*S$ in the proof of Theorem \ref{thm21} we have
$$
S^*S\overline{S\h}\subset \h_0 \oplus \h =\overline{S\h} \oplus \n(S^*).
$$
Therefore, for any $h\in \h$ there exist $k_0\in \overline{S\h}$, $k_1 \in \n(S^*)$ such that $S^*S^2h=k_0\oplus k_1$. Then
$$
S^*Sh=S^{*2}S^2h=S^*k_0,
$$
which shows that $Sh-k_0 \in \overline{S\h}\cap \n(S^*)=\{0\}$. So, by \eqref{eq26} we have
$$
k_0=Sh=J_{\h_0}(A-T^*T)^{1/2}h\oplus Th.
$$

Next, the relation $S^*S^2h=Sh\oplus k_1$ becomes
$$
k_1=(S^*S-I)Sh=(X_0^*X_0-I)(A-T^*T)^{1/2}h\oplus (A-I)Th,
$$
taking into consideration that $S^*S|_{\h_0}=X_0^*X_0$ and $S^*S|_{\h}=A$. As $h$ is an arbitrary element in $\h$, while $k_1$ depends of $h$ as above, it follows that if $S^*S\R(S)\subset \R(S)=\R(W) \oplus \overline{S\h}$, then $S^*S(Sh)=Sh$, therefore $k_1=0$. This implies $(A-I)Th=0$ for every $h\in \h$, i.e., $\R(T)\subset \n(A-I)$. 

Conversely, let's assume that $\R(T)\subset \n(A-I)$. Then $T^*AT=T^*T$ and also, $A-T^*AT=A-T^*T$. This together with the former relation in \eqref{eq25} forces to have $X_0=I_{\h_0}$. Now, by the above relation one obtains $k_1=(S^*S-I)Sh=0$ for any $h\in \h$, this meaning that $S^*S(S\h)=S\h$. Thus one has $S^*S(\overline{S\h})\subset \overline{S\h}$, and later $S^*S\R(S)\subset \R(S)=\R(W)\oplus \overline{S\h}$. Hence the conditions $S^*S\R(S)\subset \R(S)$ and $\R(T) \subset \n(A-I)$ are equivalent. 

Clearly, when these occur we have $AT=T$, whence
$$
T^{*2}T^2=T^*(T^*AT)T\le T^*AT=T^*T,
$$
taking into account that $T^*AT\le A$. So $T$ is a quasicontraction, and as $X_0=I_{\h_0}$ in this case, which implies $W^*W=S^*S|_{\h_1 \oplus \h_0}=I_{\h_1}\oplus X_0^*X_0=I_{\h_1\oplus \h_0}$, we also conclude that $W$ is an isometry.
\end{proof}


Next, as in the final part of this proof, we focus on a class of operators for which Theorem 
\ref{thm21} applies, specifically quasicontractions. In this case, the conditions \eqref{eq21} on $T$ and $A$ are stronger and always satisfied, and the corresponding lifting $S$ possesses a specific matrix structure. We present this result in the following, which is a new version of Theorem 2.3 in \cite{SS1}.

\begin{theorem} \label{thm25}
	For an operator $T \in \mathcal{B}(\h)$ the following statements are equivalent:
	\begin{enumerate}
		\item [(i)] $T$ is a quasicontraction;
	\item [(ii)] $T$ admits a left invertible quasi-isometric lifting $Q$ on a space $\mathcal{M} \supset \h$, such that $$\mathcal{M} \ominus \h \subset \n(Q^*Q-I);$$
	\item [(iii)] There exists an invertible operator $A \in \mathcal{B}(\h)$ such that 
	\begin{equation}
		T^*T=T^*AT \leq A. \label{eq27}
	\end{equation}
	\end{enumerate}
	
	Moreover, the operator $Q$ in (ii) can be chosen to be a minimal lifting for $T$.
\end{theorem}

\begin{proof}
	Let $T$ be a quasicontraction on $\h$. Then $T$ and $T^*T$ have on $\h=\overline{\mathcal{R}(T)} \oplus \n(T^*)$ the matrix representations 
	\begin{equation}
		T=\begin{pmatrix}
			C & G \\
			0 & 0
		\end{pmatrix}\text{, } \quad T^*T=\begin{pmatrix}
		C^*C & C^*G \\
		G^*C & G^*G
		\end{pmatrix}, \label{eq28}
	\end{equation}
	where $C=T|_{\overline{\mathcal{R}(T)}}$ is a contraction and $G=P|_{\overline{\mathcal{R}(T)}}T|_{\n(T^*)}$. 
	
	In our context, it is naturally to assume that $T$ is not contractive, therefore $G \neq 0$. We aim to obtain a left invertible quasi-isometric lifting for $T$, using an isometric lifting for $C$ and an invertible operator in $\mathcal{B}(\n(T^*))$.
	
	Let $D\in \mathcal{B}(\n(T^*))$ be an invertible operator with $||Dh||^2 \geq (||G||^2+\frac{1}{2})$, for $h \in \n(T^*)$ with $||h||=1$. Then $D^*D \geq (G^*G+\frac{1}{2}I)$, which subsequently yields
	$$G^*CC^*G \leq G^*G \leq \frac{1}{2}(G^*G+D^*D-\frac{1}{2} I).$$
	Hence there exists a contraction $$C_0: \overline{ \mathcal{R}(G^*G+D^*D-\frac{1}{2}I)} \to \overline{\mathcal{R}(T)}$$
	with $||C_0|| \leq \frac{1}{\sqrt{2}}$ and satisfying the relation 
	\begin{equation}
		C_0(G^*G+D^*D-\frac{1}{2}I)^{\frac{1}{2}}=C^*G. \label{eq29}
	\end{equation}
	
	Let now $D_C=(I-C^*C)^\frac{1}{2}$ and $\mathcal{D}_C=\overline{\mathcal{R}(D_C)}$ be the defect operator, and the defect space of $C$, respectively. Denote $\mathcal{M}=\mathcal{L} \oplus \h$ where $\mathcal{L}=\ell^2_+(\mathcal{D}_C \oplus \n(T^*))$, and let $Q \in \mathcal{B}(\mathcal{M})$ be the operator with the following matrix representations on $\mathcal{M}=\mathcal{L} \oplus \overline{\mathcal{R}(T)} \oplus \n(T^*)=\mathcal{L}\oplus \h$, 
	\begin{equation}
		Q=\begin{pmatrix}
			V_0 & D_0 & D_1 \\
			0 & C & G \\
			0 & 0 & 0
		\end{pmatrix}=\begin{pmatrix}
		V_0 & \widetilde{D} \\
		0 & T
		\end{pmatrix}\text{, } \quad \widetilde{D}=\begin{pmatrix}
		D_0 & D_1
		\end{pmatrix}: \h \to \mathcal{L}. \label{eq210}
	\end{equation}
	
	Here $V_0$ is the forward shift on $\mathcal{L}$, $D_0=J_0 D_C$ and $D_1=J_1 D$, where $J_0:\mathcal{D}_C \to \mathcal{L}$, $J_1:\n(T^*) \to \mathcal{L}$ are the embedding mappings, while $D$ is as above. 
	
	Since $V=Q|_{\mathcal{L} \oplus \overline{\mathcal{R}(T)}}$ is an isometry, it follows that $Q$ is a quasi-isometric lifting for $T$. We would like to state that $Q$ is left invertible and $\mathcal{M} \ominus \h \subset \n(\Delta_Q)$, where $\Delta_Q=Q^*Q-I$.
	
	Indeed, by using the above representations of $Q$ (in \eqref{eq210}), $T$ and $T^*T$ (in \eqref{eq28}), as well as the fact that $V^*_0 D_j=0$ for $j=0,1$, we obtain 
	$$Q^*Q=\begin{pmatrix}
		I & 0 & 0 \\
		0 & I & C^*G \\
		0 & G^*C & D^*D+G^*G
	\end{pmatrix}=I \oplus B,
	$$
	respectively on $\mathcal{M}=\mathcal{L} \oplus \overline{\mathcal{R}(T)} \oplus \n(T^*)=\mathcal{L} \oplus \h$, where $B=Q^*Q|_\h$.
	So $Q^*Q \mathcal{L} \subset \mathcal{L}$ and because $V_0=Q|_{\mathcal{L}}$ is an isometry, it follows that $\mathcal{M}\ominus\h=\mathcal{L} \subset \n(\Delta_Q).$
	
	On the other hand, we have on $\h=\overline{\R(T)}\oplus \n(T^*)$ that
	$$B-\frac{1}{2}I=\begin{pmatrix}
		\frac{1}{2}I & C^*G \\
		G^*C & D^*D+G^*G-\frac{1}{2}I
	\end{pmatrix}.$$
	Since the relation \eqref{eq29} can be expressed in the form
	$$C^*G=\frac{1}{\sqrt{2}}(\sqrt{2}C_0)(D^*D+G^*G-\frac{1}{2}I)^{\frac{1}{2}}$$
	where $\sqrt{2}C_0$ is a contraction, we infer from \cite[Ch. XVI, Theorem 1-1.1]{FF} for the above operator $B$, that $B-\frac{1}{2}I \geq 0$. 
	Hence $Q^*Q \geq \frac{1}{2}I$, which ensures that $Q$ is left invertible in $\mathcal{B}(\mathcal{M})$.
	
	Finally, having in view that $V_0=Q|_{\mathcal{L}}$ is a forward shift with 
	$$\n(V^*_0)=\mathcal{D}_C \oplus \n(T^*)=\overline{\mathcal{R}(D_0)} \oplus \mathcal{R}(D_1)=\overline{\mathcal{R}(\widetilde{D})},
	$$
	where $\widetilde{D}=P_{\mathcal{L}}Q|_\h$ (in \eqref{eq210}), we conclude (by \cite[Theorem 3.5]{SS1}) that $Q$ is a minimal lifting for $T$. Thus $Q$ has the properties required in assertion (ii), hence (i) implies (ii).
	
	Assume next that $Q$ on $\mathcal{M} \supset \h$ is as described in (ii). So $Q$ has on $\mathcal{M}=\h^{\prep}\oplus \h$ a matrix block as in \eqref{eq210} with respect to $T$, where $V_0=Q|_{\h^{\prep}}$ is an isometry with $V^*_0\widetilde{D}=0$.	
	Then $Q^*Q=I\oplus A$ where $A=\widetilde{D}^*\widetilde{D}+T^*T$ is an invertible operator in $\mathcal{B}(\h)$, because $Q$ is left invertible in $\mathcal{B}(M)$.
	
	It is straightforward to observe that
	 $$
	 Q^{*2}Q^2=I\oplus (\widetilde{D}^*\widetilde{D}+T^*AT)=I\oplus A=Q^*Q,
	 $$
	  taking into account that $Q$ is a quasi-isometry. It follows that $T$ and $A$ satisfy the relations $T^*T=T^*AT \leq A$, hence (ii) implies (iii).
	
	Clearly, if the preceding relations hold for some operator
	 $A$ (not necessarily invertible) then 
	 $$
	 T^{*2}T^2=T^*(T^*AT)T \leq T^*AT=T^*T,$$
	which says that $T$ is a quasicontration.
	We conclude that (iii) implies (i), and so the conditions (i), (ii) and (iii) are equivalent.
\end{proof}

It is evident that every left invertible quasi-isometry $Q$ is similar to an isometry, because in this case $Q^*Q$ is invertible and $Q^*(Q^*Q)Q=Q^*Q$.

This justifies that the lifting $Q$ from Theorem \ref{thm25} is even the natural correspondent, in the context of quasicontractions, of the minimal isometric lifting of a contraction. To be more precise, we have in view that $Q$ is a quasi-isometry similar to an isometry and a minimal lifting for the quasicontraction $T$ on $\h$, such that $Q|_{\h^{\prep}}$ is an isometry with $Q^*Q\h \subset \h$.

In the case when $T$ in \eqref{eq28} is a quasi-isometry, then $C=T|_{\overline{\mathcal{R}(T)}}$ is an isometry, therefore $D_0=0$ and $V_0$ is the foward shift on $\mathcal{L}=\ell^2_+(\n(T^*))$, in the matrix \eqref{eq210} of the lifting $Q$.
In this case, $Q\overline{\mathcal{R}(T)} \subset \overline{\mathcal{R}(T)}$, so the lifting $Q$ for $T$ is obtained by an extension of the isometry $T|_{\overline{\mathcal{R}(T)}}$.
We mention this special case in the following

\begin{corollary}
	\label{c26}
	Every quasi-isometry $T$ on $\h$ admits a minimal quasi-isometric lifting $Q$ on $\mathcal{M} \supset \h$, which is similar to an isometry, such that $\mathcal{M} \ominus \h \subset \n(Q^*Q-I)$ and $Q\overline{\mathcal{R}(T)}\subset\overline{\mathcal{R}(T)}$.

\end{corollary}

	The following result can be useful in applications, in order to get operators with natural quasi-isometric liftings, which are similar to isometries.
	
	\begin{theorem}
		\label{thm27} Let $A$, $T \in \mathcal{B}(\h)$, $A$ being invertible, such that $T^*T \leq T^*AT \leq A$, and let $\widehat{T}\in \mathcal{B}(\h)$ be the contraction defined by the relation $\widehat{T}A^{1/2}=A^{1/2}T$.
		
		If $\n_0 := \n(A-T^*AT)$ is invariant for $T$, then $\n_0=\n(I-\widehat{T}^*\widehat{T})=:\n_1$ and this subspace is also invariant for $A$ and $\widehat{T}$. Conversely, if $\n_1$ is invariant for $\widehat{T},$ then $\n_1=\n_0$ and this subspace is also invariant for $A$ and $T$.
		
		Moreover, if $T\n_0 \subset \n_0$ then the following statements are equivalent:
		\begin{enumerate}
			\item [(i)] $\mathcal{R}(A-T^*AT)=\mathcal{R}(A-T^*T)$ and this range is closed;
			\item [(ii)] $T$ has the matrix representation
		\end{enumerate}
			\begin{equation} \label{eq211}
				T=\begin{pmatrix}
					W & T_0 \\
					0 & T_1
				\end{pmatrix} \text{ on } \h=\n_0 \oplus (\h \ominus \n_0),
			\end{equation}
			where $W^*W=A_0$, $\mathcal{R}(T_0) \subset \n(W^*)$ and
			$$||\begin{pmatrix}
				A^{1/2}_0T_0A^{-1/2}_1 & A^{1/2}_1T_1A^{-1/2}_1
			\end{pmatrix}^{\rm tr}||<1\text{, } \quad A_0=A|_{\n_0}\text{, } \quad A_1=A|_{\h \ominus \n_0}.$$
	
		Additionally, when the statement (i) holds, $T$ admits a natural quasi-isometric lifting that is similar to an isometry, and also $W$ in \eqref{eq211} is a quasi-isometry similar to an isometry.
	\end{theorem}
	 
	 \begin{proof}
	 	Assume that
	 	 $T\n_0\subset \n_0$. Since $\n(A)=\{0\}$, from Proposition 2.1 and Theorem 4.6 in \cite{S-2006}, it follows that $A\n_0 \subset \n_0=\n(I-S_{\widehat{T}})$, where $S_{\widehat{T}}$ is the asymptotic limit of the contraction $\widehat{T}$, defined as 
	 	$$S_{\widehat{T}}=\lim\limits_{n \to \infty}\widehat{T}^{*n}\widehat{T}^n, \quad \text{ strongly in } \mathcal{B}(\h).$$
	 	
	 	Now, from the definition of $\widehat{T}$ we have $A-T^*AT=A^{1/2}(I-\widehat{T}^*\widehat{T})A^{1/2}$.
	 	Thus, if $x \in \n_1=\n(I-\widehat{T}^*\widehat{T})$ then $A^{-1/2}x \in \n_0$, and later $x\in\n_0$ since $A\n_0\subset\n_0$ (and $A$ is invertible).
	 	Therefore $\n_1 \subset \n_0=\n(I-S_{\widehat{T}}) \subset \n(I-\widehat{T}^*\widehat{T})=\n_1$, having in view that $\n(I-S_{\widehat{T}})$ is invariant for $\widehat{T}$ and $\widehat{T}$ is an isometry on this subspace.
	 	We conclude that $\n_0=\n_1$ and this subspace is invariant for $A$, $T$ and $\widehat{T}$.
	 	
	 	Conversely, let us assume that $\widehat{T}\n_1 \subset \n_1$. 	
	 	So for $h \in \n_1$ we have $h=\widehat{T}^{*n}\widehat{T}^nh$ for any integer $n \geq 1$.
	 	Hence 
	 	$$x \in \bigcap\limits_{n \geq 1}\n(I-\widehat{T}^{*n}\widehat{T}^n)=\n(I-S_{\widehat{T}}),$$
	 	 which ensures that $\n_1=\n(I-S_{\widehat{T}})$, while by \cite[Theorem 4.6]{S-2006} this subspace is invariant for $A$, and so it is also invariant for $A^{-1}$. This shows that $h \in \n_1$ if and only if $A^{1/2}h \in \n_1$, this meaning that $h \in \n_0$.
	 	Hence $\n_0=\n_1$, and since this subspace is invariant for $A$ and $\widehat{T}$, it is also invariant for $T$, because by the above relation between $A$, $T$ and $\widehat{T}$, for $h \in \n_0$ we have $Th \in \n_0$ if and only if $A^{1/2}Th=\widehat{T}A^{1/2}h \in \n_1$.
	 	This concludes the first assertion of theorem.
	 	
	 	For the second statement, assume that $T \n_0 \subset \n_0$. 
	 	Therefore $T$ has under the decomposition $\h=\n_0 \oplus (\h \ominus \n_0)$ a matrix representation of the form \eqref{eq211}, with the appropriate entries $W$, $T_0$, $T_1$.
	 	
	 	Since $A \n_0 \subset \n_0$ and $A(\h \ominus \n_0)\subset \h\ominus \n_0$, we have $A=A_0 \oplus A_1$ on $\h=\n_0 \oplus (\h \ominus \n_0)$.
	 	So we obtain
	 	\begin{eqnarray} \label{eq212}
	 		A-T^*AT&=& \begin{pmatrix} 
	 			A_0-W^*A_0W & -W^*T_0 \\
	 			-T^*_0W & A_1-T^*_0A_0T_0-T^*_1A_1T_1
	 		\end{pmatrix}\\
	 		&=& \begin{pmatrix}
	 		0 & 0 \\
	 		0 & A_1-T^*_0A_0T_0-T^*_1A_1T_1
	 		\end{pmatrix}, \notag
	 	\end{eqnarray}
	 	taking into consideration that $A-T^*AT \geq 0$ and $(A-T^*AT)|_{\n_0}=0$, which forces $W^*T_0=0$, that is $T^*T \n_0 \subset \n_0$.
	 	This also shows that 
	 	\begin{equation} \label{eq213}
	 		A-T^*T=\begin{pmatrix}
	 			A_0-W^*W & 0 \\
	 			0 & A_1-T^*_0T_0-T^*_1T_1
	 		\end{pmatrix}.
	 	\end{equation}
	 	
	 	On the other hand, since $\n_1=\n_0$ and it is an invariant subspace for $\widehat{T}$ (as we seen before), $\widehat{T}$ has a representation of the form
	 	\begin{equation} \label{eq214}
	 		\widehat{T}=\begin{pmatrix}
	 			V & C_0 \\
	 			0 & C_1	 		\end{pmatrix} \text{ on } \h=\n_0 \oplus(\h \ominus \n_0),
	 	\end{equation}
	 	where $V$ is an isometry and $C_0$, $C_1$ are contractions. 
	 	Furthermore, $V^*C_0=0$ because $\widehat{T}$ is a contraction. 
This subsequently gives
	 	\begin{equation} \label{eq215}
	 		I-\widehat{T}^*\widehat{T}=\begin{pmatrix}
	 			0 & 0 \\
	 			0 & I-C^*_0C_0-C_1^*C_1
	 		\end{pmatrix},
	 	\end{equation}
	 and later we obtain
	 	\begin{equation}
	 		\label{eq216} A-T^*AT=A^{1/2}(I-\widehat{T}^*\widehat{T})A^{1/2}=\begin{pmatrix}
	 			0 & 0 \\
	 			0 & A_1^{1/2}(I-C^*_0C_0-C_1^*C_1)A_1^{1/2}
	 		\end{pmatrix}.
	 	\end{equation}
	 	
	 	Now assume that the statement (i) holds. Then $\n_0=\n(A-T^*T)$, so $W^*W=A_0$ in \eqref{eq213}, while by \eqref{eq212} one has $\h \ominus \n_0=\mathcal{R}(A-T^*AT)$.
	 	Hence \begin{equation} \label{eq217}
	 		(A-T^*AT)|_{\h \ominus \n_0}=A_1^{1/2}[I-A_1^{-1/2}(T^*_0A_0T_0+T^*_1A_1T_1)A^{-1/2}_1]A_1^{1/2}
	 	\end{equation}
	 	is an invertible operator, and the same is true for the operator $A^{-1/2}(A-T^*AT)A^{-1/2}|_{\h \ominus \n_0}.$
	 	
	 	This implies that $C=A^{-1/2}_1(T^*_0A_0T_0+T^*_1A_1T_1)A^{-1/2}_1$ is a strict contraction, meaning that 
	 	$$ 
	 	\| \begin{pmatrix}
A^{1/2}_0T_0A^{-1/2}_1 & A^{1/2}_1T_1A_1^{-1/2}
	 	\end{pmatrix}^{\rm tr} \|<1.
	 	$$ 
	 	Thus (i) implies (ii).
	 	
	 	For the converse implication, assume $||C||<1$ and $W^*W=A_0$.
	 	This later with \eqref{eq213} and the relations $T^*T \leq T^*AT \leq A$ imply $\n_0=\n(A-T^*T)$, so $\h \ominus \n_0=\overline{\mathcal{R}(A-T^*T)}$. But the condition $||C||<1$ ensures that $I-C^*C$ is an invertible operator in $\mathcal{B}(\h \ominus \n_0)$, which by \eqref{eq217} means that $(A-T^*AT)|_{\h \ominus\n_0}$ is invertible, too.
	 	Hence 
	 	$$\mathcal{R}(A-T^*AT)=\h \ominus \n_0=\overline{\mathcal{R}(A-T^*T)}.$$
	 	
	 	Since $A-T^*AT \leq A-T^*T$ (from the relations $T^*T\le T^*AT \le A$), we have $\R(A-T^*AT)\subset \R[(A-T^*T)^{1/2}]$, which together with the previous relation implies that $\R(A-T^*T)$ is closed, and finally that $\mathcal{R}(A-T^*AT)=\mathcal{R}(A-T^*T)$, this meaning the statement (i). Hence (ii) implies (i).
	 	
	 	The last assertion in theorem is derived from Theorem \ref{thm21}, because from (i) one obtains $\mathcal{R}[(A-T^*AT)^{1/2}]=\mathcal{R}[(A-T^*T)^{1/2}]$. 
	 	Also, in this case $W$ from \eqref{eq212} and \eqref{eq213} satisfies the conditions $W^*W=A_0=W^*A_0W=W^{*2}W^2$, whence we conclude that $W$ is a quasi-isometry similar to an isometry.
	 \end{proof}
	 
	 A very special case of Theorem $\ref{thm27}$ is now mentioned.
	 
	 \begin{corollary} \label{c28}
	 	For $T \in \mathcal{B}(\h)$ there exists an invertible operator $A \in \mathcal{B}(\h)$ such that $T^*T \leq T^*AT \leq A$ and $\mathcal{R}(A-T^*AT)=\h=\mathcal{R}(A-T^*T)$, if and only if the spectral radius of $T$ is strictly less than $1$.
	 \end{corollary}
	 
	 \begin{proof}
	 	If $T$ and $A$ are as above then $\n_0=\n(A-T^*AT)=\{0\}$, so $T=T_1$ in \eqref{eq211} and $||A^{1/2}TA^{-1/2}||<1$.
	 	Hence $T$ is similar to a strict contraction, which by C. Rota's result means that $r(T)<1$ ($r$ being the spectral radius). 
	 	
	 	Conversely, if $r(T)<1$ then $T$ is similar to a contraction $\widehat{T}$ with $||\widehat{T}||<1$. 
	 	So, there exists an invertible operator $A$ with $T^*T \leq T^*AT \leq A$, such that $A^{1/2}T=\widehat{T}A^{1/2}$.
	 	
	 	Since $A-T^*AT=A^{1/2}(I-\widehat{T}^*\widehat{T})A^{1/2}$ and because $I-\widehat{T}^*\widehat{T}$ is invertible (as $A$), it follows that $\mathcal{R}(A-T^*AT)=\h$. But $A-T^*AT \leq A-T^*T$ which implies 
	 	$$
	 	\mathcal{R}[(A-T^*AT)^{1/2}] \subset \mathcal{R}[(A-T^*T)^{1/2}],
	 	$$ 
	 	and finally one obtains $\h=\mathcal{R}(A-T^*AT)=\mathcal{R}(A-T^*T)$.	
	 	 \end{proof}
	 	 	
	 		\begin{remark}\label{re29}
	 		\rm
	 		Each operator $T$ on $\h$ with $r(T)<1$ (as above) is a $\rho$-contraction in the sense of Sz.-Nagy-C. Foias \cite{SFb} for some constant $\rho >0$. 
	 		This means that $T$ admits a unitary $\rho$-dilation $U$ on a space $\mathcal{M} \supset \h$, that is $T^n=\rho P_{\h}U^n|_{\h}$ ($n \geq 1$), where $U$ is a unitary operator on $\m$. So $T=P_{\h}(\rho U)|_{\h}$ is a compression of the invertible operator $\rho U$, but it is neither quasi-isometric, nor a lifting for $T$ (when $\rho \neq 1$).
	 	 
	 	 	By comparison, the above corollary demonstrates that certain $\rho$-contractions possess natural quasi-isometric liftings, which are similar to isometries. This fact offers a resemblance with the isometric liftings, for the operators $T$ with $r(T)<1$. 

On the other hand, we also have operators $T$ with $r(T)=1$, which admit such liftings, as are the quasi-isometries (from Corollary \ref{c26}), as well as those of the form \eqref{eq211} with $\n_0=\n(A-T^*AT)\neq \{0\}$.	
\end{remark}

	At the end of the next section, we will mention a class of $2$-quasi-isometries for which Theorem \ref{thm27} can be applied.

\medskip

\section{General left invertible quasi-isometric liftings and applications}	
	\medskip

The above results just refer to a class of operators which have left invertible and natural quasi-isometric liftings. But not every operator similar to a contraction admits such a lifting (as we see below). However, the following result holds true.

\begin{theorem}\label{thm31}\quad
	
\begin{itemize}
\item[(i)] Every operator $T \in \mathcal{B}(\h)$ similar to a contraction admits a quasi-isometric lifting $S$ on a space $\ka \supset \h$ which is similar to an isometry, such that $S$ has under a decomposition $\ka=\mathcal{L}\oplus \mathcal{M}$ with $\{0\}\neq \mathcal{L} \subset \ka \ominus \h$, a block matrix of the form 
		\begin{equation}\label{eq31}
			S=\begin{pmatrix}
				V & G \\
				0 & Q
			\end{pmatrix}.
		\end{equation}
		Here $V=S|_{\mathcal{L}}$ is an isometry with $\n(V^*)=\mathcal{R}(G)$, while $Q$ is a quasi-isometric lifting for $T$ on the space $\mathcal{M} \supset \h$, with $\overline{\mathcal{R}(Q)}=\n(G)$ and $\n(Q^*)\subset \n(S^*)$.
	
\item [(ii)] For every left invertible quasi-isometric lifting $S$ for $T$ of the form \eqref{eq31} on $\ka=\mathcal{L} \oplus \mathcal{M}$, where $S \mathcal{L} \subset \mathcal{L} \subset \ka \ominus \h$ and $V=S|_{\mathcal{L}}$ is an isometry with $\n(V^*)=\overline{\mathcal{R}(G)}$, we have that the lifting $Q=P_{\m}S|_{\m}$ for $T$ on $\mathcal{M} \supset \h$ is quasi-isometric if and only if $\mathcal{R}(Q)\subset \n(G).$ \newline
    Additionally, the following conditions
	$$\mathcal{R}(G^*) \subset \n(Q^*)=\n(S^*) \text{ and } \quad S\mathcal{R}(Q)\subset\mathcal{R}(Q) \subset \mathcal{R}(S)$$
	are simultaneously satisfied.\end{itemize}
\end{theorem}

\begin{proof}
	Let $T \in \mathcal{B}(\h)$ be similar to a contraction, and $Q$ on $\mathcal{M} \supset \h$ be a quasi-isometric lifting for $T$.
	Then $Q$ and $Q^*Q$ have on $\mathcal{M}=\overline{\mathcal{R}(Q)}\oplus \n(Q^*)$ the matrix representations
	\begin{equation}
		\label{eq32}
		Q=\begin{pmatrix}
			V_1 & G_1 \\
			0 & 0
		\end{pmatrix}\text{, } \quad Q^*Q=\begin{pmatrix}
		I & V^*_1G_1 \\
		G^*_1V_1 & G^*_1G_1
		\end{pmatrix}\text{, }
	\end{equation}
	where $V_1$ is an isometry.
Supposing that $T$ is not a contraction (this being a trivial case in our context), we have $G_1 \neq 0$.

Now, by Corollary \ref{c26}, $Q$ possesses a left invertible quasi-isometric lifting $S$ on a space $\ka=\mathcal{L} \oplus \mathcal{M}$, where $$\mathcal{L}=\ell^2_+(\n(Q^*)) \subset \n(S^*S-I).$$

Then, under the decompositions
$$\ka=\mathcal{L}\oplus \overline{\mathcal{R}(Q)} \oplus \n(Q^*)=(\ka \ominus \h)\oplus \h=\mathcal{L}\oplus \mathcal{M},$$
$S$ has, respectively, the matrix representations 
\begin{equation}
	\label{eq33}
	S=\begin{pmatrix}
		V&0 & G_0 \\
		0 & V_1 & G_1 \\
		0 & 0 &0  
	\end{pmatrix}=\begin{pmatrix}
	W & E \\
	0 & T
	\end{pmatrix}=\begin{pmatrix}
	V & G \\
	0 & Q
	\end{pmatrix}\text{,}
\end{equation}
with the corresponding appropriate entries. 
More precisely, $V$ is the foward shift on $\mathcal{L}$, while $G_0:\n(Q^*) \to \mathcal{L}$ is a left invertible operator with $\mathcal{R}(G_0)=\mathcal{R}(G)=\n(V^*)$, where $G= \begin{pmatrix}
	0 & G_0
\end{pmatrix}$ acts as an operator from $\mathcal{M}=\overline{\R(Q)} \oplus \n(Q^*)$ into $\mathcal{L}$, while $E:\h \to \ka \ominus \h$. 

To compare with the construction of $Q$ in \eqref{eq210} we have in \eqref{eq33} that $V_1$ is an isometry, so $D_{V_1}=0$ (instead of $C$, respectively $D_C$) in \eqref{eq210}.

Now from \eqref{eq33} we obtain that $\overline{\mathcal{R}(Q)}=\n(G)$, taking into account that $G_0$ is injective. In addition to this, we have that $\n(Q^*) \subset \n(S^*)$, because $Q^*=S^*|_{\m}$, which completes the proof of statement (i). 

For the assertion (ii), we consider $S$ to be a left invertible quasi-isometric lifting for $T$, having a block matrix in $\ka=\mathcal{L}\oplus \mathcal{M}$ as in \eqref{eq31}, where $\{0\}\neq \mathcal{L}=\ka \ominus \mathcal{H}$, such that $S\mathcal{L} \subset \mathcal{L}$ and $V=S|_\mathcal{L}$ is an isometry with $\n(V^*)=\overline{\mathcal{R}(G)}$, $G=P_\mathcal{L}S|_\mathcal{M}$.

Clearly, this ensures that $Q$ is a quasicontractive lifting for $T$ on the space $\mathcal{M} \supset \h$, taking into account that $S$ is a lifting for $Q$ (see for instance, \cite[Theorem 2.3]{SS1}). Furthermore, the relation $S^*S|_{\m}=S^{*2}S^2|_{\m}$ obtained through the last representation of $S$ in \eqref{eq33}, where $V^*G=0$, means in the terms of $G,Q$ that 
$$G^*G+Q^*Q=G^*G+Q^*(G^*G+Q^*Q)Q,$$
that is $Q^*Q=Q^*G^*GQ+Q^{*2}Q^2$.

Hence $Q$ is a quasi-isometry if and only if $GQ=0$, meaning that $\mathcal{R}(Q) \subset \n(G)$.
This gives the first assertion in (ii).

For the second statement in (ii), we suppose that
 $\n(S^*)=\n(Q^*) \supset \mathcal{R}(G^*)$. 
Then $\mathcal{R}(S)=\mathcal{L}\oplus \overline{\mathcal{R}(Q)}$, while for $k \in\overline{\mathcal{R}(Q)}$ and $\ell \in \mathcal{L}$ we have (using \eqref{eq31})
$$(Sk,\ell)=(k, S^*\ell)=(k,V^*\ell \oplus G^*\ell)=(k, G^*\ell)=(Gk, \ell)=0,$$
since $\mathcal{R}(Q) \subset \n(G)$ by our prior assumption.
So $Sk$ is orthogonal on $\mathcal{L}$ for any $k\in \overline{\R(Q)}$, consequently $S\overline{\mathcal{R}(Q)} \subset \overline{\mathcal{R}(Q)} \subset \mathcal{R}(S)$.

Conversely, assume that these last inclusions hold.
Then $\n(S^*) \subset \mathcal{L} \oplus \n(Q^*)$ and for $k=k_0\oplus k_1 \in \n(S^*)$ with $k_0 \in \mathcal{L}$ and $k_1\in \n(Q^*)$, we have (by \eqref{eq31}) that $k_0 \in \n(V^*)\cap \n(G^*)=\{0\}$, having in view that $\n(G^*)=\mathcal{R}(V)$ by our assumption in (ii).
So $k=k_1 \in \n(Q^*)$, which concludes that $\n(S^*)=\n(Q^*)$.

In addition, since $S\mathcal{R}(Q) \subset \overline{\mathcal{R}(Q)}$ we have that $G\mathcal{R}(Q)=P_\mathcal{L}S\mathcal{R}(Q)=\{0\}$, according to the representation in \eqref{eq31}.
So $\mathcal{R}(Q) \subset \n(G)$, which finally shows that $\mathcal{R}(G^*) \subset \n(Q^*)=\n(S^*)$. All statements are now proved.
\end{proof}

\begin{remark}\label{re32}
	\rm
	The lifting $S$ from Theorem \ref{thm31} (i) with the representation \eqref{eq31}, is not necessary a minimal lifting for $T$, but it is a minimal lifting for $Q$. 
	For the later assertion, we recall from the construction of $S$, that $V$ is the foward shift on $\mathcal{L}$ with $\n(V^*)=\mathcal{R}(G)$.
	
	However, since the condition $S^*S \geq c I$, for a constant $c>0$, is preserved for the restriction of $S$ to each invariant subspace for $S$, we can conclude the following result.
\end{remark}

\begin{corollary}
	\label{c33} Every operator similar to a contraction admits a minimal quasi-isometric lifting, which is similar to an isometry.
\end{corollary}

	Notice that for such a minimal lifting, the relationship between entries of the block matrix from \eqref{eq31} are not preserved, in general.
	
	Concerning the statement (ii) in Theorem \ref{thm31}, it is clear that $S$ can be always considered a minimal lifting for $Q$, which does not affect the equivalences mentioned here.
	
	Returning to the statement (i), we remark from the second representation of $S$ in \eqref{eq33} that $W^*E \neq 0$, i.e. $S^*S \h \not \subset \h$ in general, even if $Q$ is a minimal lifting for $T$. On the other hand, if $Q$ satisfies the kernel condition $Q^*Q\n(Q^*) \subset \n(Q^*)$, then $S^*S \n(Q^*)\subset \n(Q^*)$ by the representation \eqref{eq33} of $S$. But this does not mean the kernel condition for $S$, having in view that the inclusion $\n(Q^*)\subset \n(S^*)$ is strict, in general. More precisely, we can even determine the subspace $\n(S^*)\ominus \n(Q^*)$ under the kernel condition of a quasi-isometry $Q$ (not necessary as a lifting of an operator $T$). 
	
	\begin{proposition}\label{p342}
	Let $Q\in \mathcal{B}(\m)$ be a non-contractive quasi-isometry with the kernel condition $Q^*Q \n(Q^*) \subset \n(Q^*)$ and $S$ on $\ka =\mathcal{L}\oplus \overline{\R(Q)}\oplus \n(Q^*)$ be the left invertible quasi-isometric lifting for $Q$, with the block matrix \eqref{eq33} and its entries $V,V_1,G_0,G_1$. Then
	\begin{eqnarray}\label{eq342}
	\n(S^*)\ominus \n(Q^*) &=& \{l\oplus m:l\in \n(V^*),\quad m\in \n(V_1^*),\quad G_0^*l+G_1^*m=0\}\\
	&=& \n \left[\begin{pmatrix} G_0^* & G_1^* \end{pmatrix}\right] \ominus (\R(V)\oplus \n(G_1^*)). \notag
	\end{eqnarray}
	\end{proposition}
	
\begin{proof}
We know that $\{0\}\neq \n(Q^*)\subset \n(S^*)$ because $Q^*=S|_{\m}$. But the kernel condition $Q^*Q\n(Q^*)\subset \n(Q^*)$ ensures that $\n(V_1^*)=\overline{\R(G_1)}$, so $\R(V_1)=\n(G_1^*)$, where $V_1=Q|_{\overline{\R(Q)}}$ is an isometry and $G_1=Q|_{\n(Q^*)}$. 

On the other hand, we have $\{0\}\neq \R \left[\begin{pmatrix} G_0& G_1 \end{pmatrix}^{\rm tr} \right] \subset \R(V)\oplus \n(G_1^*)$, where $V=S|_{\mathcal{L}}$ is an isometry with $\n(V^*)=\R(G_0)$, while $G_0:\n(Q^*)\to \mathcal{L}$ is a left invertible operator. 

Now, it is clear that $\n(S^*)\ominus \n(Q^*) \subset \mathcal{L}\oplus \overline{\R(Q)}$, and the former expression of $\n(S^*)\ominus \n(Q^*)$ immediately follows from the representation \eqref{eq33} of $S$. Thus, if $k=l\oplus m$ with $l\in \n(V^*)$ and $m\in \n(V_1^*)=\overline{\R(G_1)}$ such that $G_0^*l+G_1^*m=0$, then $k\in  \R \left[\begin{pmatrix} G_0^* & G_1^* \end{pmatrix} \right]$ and $l\perp \R(V)$, $m\perp \n(G_1^*)$, hence $k$ is orthogonal on $\R(V)\oplus \n(G_1^*)$. This gives the inclusion
$$
\n(S^*)\ominus \n(Q^*)\subset \n \left[\begin{pmatrix} G_0^* & G_1^* \end{pmatrix} \right] \ominus (\R(V)\oplus \n(G_1^*)).
$$

Conversely, let $k=l\oplus m \in \n \left[\begin{pmatrix} G_0^* & G_1^* \end{pmatrix} \right] $ and $k\perp \R(V)\oplus \n(G_1^*)$. In particular, $k\perp VV^*l\oplus m_1$ where $m_1=P_{\n(G_1^*)}m$, that is
$$
0=(l\oplus m,VV^*l\oplus m_1)=\|V^*l\|^2+\|m_1\|^2.
$$
It follows that $V^*l=0$, i.e., $l\in \n(V^*)$ and $m_1=0$, hence $m\in \overline{\R(G_1)}=\n(V_1^*)$. But this means that $k\in \n(S^*)\ominus \n(Q^*)$, and so we get the inclusion 
$$
\n \left[\begin{pmatrix} G_0^* & G_1^* \end{pmatrix} \right] \ominus (\R(V)\oplus \n(G_1^*)) \subset \n(S^*)\ominus \n(Q^*),
$$ 
and finally the second relation in \eqref{eq342}.
\end{proof}	
	
	The above observations indicate the existence of operators similar to contractions that cannot admit quasi-isometric liftings which are simultaneously left invertible and natural. This limitation holds even for a particular class of operators, as illustrated in the following example.
	
	\begin{example}\label{ex34}
		Let $T\in \mathcal{B}(\h)$ be a non-contractive operator similar to a contractive symmetry $J \in \mathcal{B}(\h)$. The latter means that $J^2=I$ and $\|J\|\le 1$, which forces $J$ to be unitary.
		
	Therefore exists a positive invertible operator $A_0 \in \mathcal{B}(\h)$ such that $A^{1/2}_0T=JA_0^{1/2}$. Then $A^{1/2}_0T^2=JA_0^{1/2}T=J^2A^{1/2}_0=A_0^{1/2}$, which implies $T^2=I$.
		
		We use Theorem \ref{thm21} to show that $T$ has not a left invertible and natural quasi-isometric lifting. 
			More precisely, we see that there no invertible operator $A\in \mathcal{B}(\h)$ which satisfies the conditions \eqref{eq21} relative to $T$.
		
		Indeed, let $A$ be any invertible operator in $\mathcal{B}(\h)$ such that $T^*T \leq A$ and $T^*AT \leq A$.
		So there is a contraction $C$ on $\h$ with $CA^{1/2}=A^{1/2}T$.
		Then $C^2A^{1/2}=A^{1/2}T^2=A^{1/2}$ which implies $C^2=I$.
		Since $||C|| \leq 1$ it follows that $C$ is unitary, which later yields that $T^*AT=A$.
			Thus, we get on one hand $\mathcal{R}(A-T^*AT)=\{0\}$, i.e., $\n(A-T^*AT)=\h$.
		
		Assume now that  $\n(A-T^*T)=\h$, which means $T^*T=A$.
		Then from this relation and the one obtained before, we infer that
		$$T^*T=A=T^*AT=T^{*2}T^2=I$$
		because $T^2=I$.
		But this contradicts our assumption that $T$ is not a contaction.
		
		We conclude that  $\n(A-T^*T) \not \subset \h$, or equivalently $$\mathcal{R}[(A-T^*T)^{1/2}]\not = \{0\}=\mathcal{R}[(A-T^*AT)^{1/2}],$$
		and this happens for any operator $A$ as has been chosen above. 
		
		In other words, $T$ does not satisfy the conditions \eqref{eq21}, which proves that $T$ does not admit a left invertible and natural quasi-isometric lifting.
\end{example}
		
		A concrete example of an operator $T$ as described above, is now presented.
		
\begin{example}\label{ex35}
		
		On the space $\widetilde{\h}=\h \oplus \h$ we consider the operators $J$, $A_0$ and $T$, having the block matrices
		$$J=\begin{pmatrix}
			0 & J_0 \\
			J^*_0 & 0 
		\end{pmatrix}\text{, }
		A_0=\begin{pmatrix}
			\frac{1}{4}I & 0 \\
			0 & I
		\end{pmatrix}\text{, } T=A_0^{-1/2}JA_0^{1/2}=\begin{pmatrix}
		0 & 2J_0 \\
		\frac{1}{2}J_0^* & 0
		\end{pmatrix}\text{, }$$
		where $J_0:\{0\}\oplus \h \to \h \oplus \{0\}$ is the natural embedding.
		
		Clearly, $J^2=I$ with $||J||=1$ and $A^{1/2}_0T=JA_0^{1/2}$, therefore $T$ is a non-contractive operator which is similar with the unitary operator $J$, by the invertible operator $A_0$ into $\mathcal{B}(\widetilde{\h})$. But by the arguments from the previous example, for any invertible operator $A\in \mathcal{B}(\widetilde{\h})$ with $T^*T\le A$ and $T^*AT\le A$, together $T$ and $A$ cannot satisfy the range condition from \eqref{eq21}. Hence for every left invertible quasi-isometric lifting $S$ for $T$ (which is assured by Theorem \ref{thm31}), one has $S^*S\h \not\subset \h$. 
	\end{example}
	
	Next we present another result in relationship to Theorem \ref{thm27} and Theorem \ref{thm31}, which refers to some classes of operators with left invertible quasi-isometric liftings, but not necessarily natural liftings in this context.
	
	 	 \begin{proposition}
	 	 	\label{p36}
	 	 	Let $A$, $T$, $\widehat{T} \in \mathcal{B}(\h)$ be as in the hypothesis of Theorem \ref{thm27}, such that the subspace $\n_0=\n(A-T^*AT)$ is invariant for $T$. Then 
	 	 	\begin{enumerate}
	 	 		\item [(i)] $\mathcal{R}(A-T^*AT)$ is closed if and only if $||\widehat{T}|_{\h \ominus \n_0}||<1$.
	 	 		\item [(ii)] If $\n_0=\n(A-T^*T)$, then $\mathcal{R}(A-T^*T)$ is closed if and only if $||TA^{-1/2}|_{\h \ominus \n_0}||<1$.
	 	 	\end{enumerate}
	 	 \end{proposition}
	 	 
	 \begin{proof}
	 	Since $T \n_0 \subset \n_0$ we have also $\widehat{T} \n_0 \subset \n_0$, so $T$ and $\widehat{T}$ have the representation \eqref{eq211} and \eqref{eq214}, respectively, on $\h=\n_0 \oplus \overline{\mathcal{R}(A-T^*AT)}$.
	 	
	 	To show (i) assume that $\mathcal{R}(A-T^*AT)$ is closed. 
	 	Since $\n_0=\n(I-\widehat{T}^*\widehat{T})$, by Theorem \ref{thm27}, it follows from the representation \eqref{eq215} that $I-C^*_0C_0-C_1^*C_1$ is invertible in $\mathcal{B}(\h \ominus\n_0)$, where $\h\ominus \n_0=\R(A-T^*AT)$.	 	
	 	But this implies that $||C^*_0C_0+C^*_1C_1||<1$, which means $||\widehat{T}|_{\h \ominus \n_0}||<1$.
	 	
	 	Conversely, the last condition ensures that $I-C^*_0C_0-C_1^*C_1$, as well as $(A-T^*AT)|_{\h \ominus \n_0}$ in \eqref{eq217}, are invertible operators in $\mathcal{B}(\h\ominus \n_0)$. Hence
	 	$$\overline{\mathcal{R}(A-T^*AT)}=\h\ominus\n_0=(A-T^*AT)(\h \oplus\n_0)=\mathcal{R}(A-T^*AT).$$
	 	 Thus the equivalence in (i) is proved.
	 	
	 	For (ii) we proceed similarly, assuming that $\n_0=\n(A-T^*T)$.	 	
	 	Thus, we have from hypothesis that $A^{1/2}T=\widehat{T}A^{1/2}$, whence $A-T^*T=A^{1/2}(I-A^{-1/2}T^*TA^{-1/2})A^{1/2}$.
	 	
	 	Now, if $\mathcal{R}(A-T^*T)$ is closed then $\h \ominus \n_0=\mathcal{R}(A-T^*T)$ and $(A-T^*T)|_{\h \ominus\n_0}$ is invertible.
	 	 Thus, from the previous relation we obtain that $||TA^{-1/2}|_{\h \ominus\n_0}||<1$. 
	 	 
	 	 Conversely, this last condition shows that $(A-T^*T)|_{\h\ominus \n_0}$ is invertible in 	 	 $\mathcal{B}(\h \ominus \n_0)$, hence 
	 	 $$
	 	 \overline{\mathcal{R}(A-T^*T)}=\h \ominus\n_0=(A-T^*T)(\h \ominus\n_0)=\mathcal{R}(A-T^*T).
	 	 $$
	 	 This concludes the assertion (ii).
	 \end{proof}
	 Finally, we observe that the assertions (i) and (ii) in this proposition are more general than the assertion (i) in Theorem \ref{thm27}, although they are related to each other. An interesting class of operators for which the ranges from the assertions (i) and (ii) coincide will be presented at the end of this paper, as an application of Theorem \ref{thm27}. 
	 
	 \begin{theorem}\label{thm37}
	 	Let $T \in \mathcal{B}(\h)$ be such that $T^*T \leq T^{*2}T^2=T^{*3}T^3$ and $T^*T \n(T^*) \subset \n(T^*)$. Then there exists an invertible operator $A\in \B$ with $T^*T\le T^*AT\le A$ such that $\R(A-T^*T)=\R(A-T^*AT)$ and this range is closed, hence it admits a natural quasi-isometric lifting, which is similar to an isometry.
	 	\end{theorem}
	 	
	 	\begin{proof}
	 	Let $W=T|_{\overline{\mathcal{R}(T)}}$. Then $W^*W=P_{\overline{\mathcal{R}(T)}}T^*T|_{\overline{\mathcal{R}(T)}}$, so for $h \in \h$ the relations 
	 	$$
	 	(Th, Th)\leq (T^*TTh, Th)=(T^{*2}T^2Th,Th)
	 	$$
	 	 expressed in terms of $W$ become $$I \leq W^*W=W^{*2}W^2.$$
	 	
	 	Therefore $W$ is an expansive quasi-isometry, while $T$ an $T^*T$ have on $\h=\overline{\mathcal{R}(T)} \oplus \n(T^*)$ the representations
	 	\begin{equation} \label{eq34}
	 		T=\begin{pmatrix}
	 			W & T_0 \\
	 			0 & 0
	 		\end{pmatrix}\text{, } \quad T^*T=\begin{pmatrix}
	 		W^*W & 0 \\
	 		0 & T^*_0T_0
	 		\end{pmatrix}
	 	\end{equation}
	 	where $W^*T_0=0$, meaning even the condition $T^*T\n(T^*)\subset \n(T^*)$ from our assumption.
	 	
	 	Clearly, one may suppose $T_0 \neq 0 $ (that is $T$ is not quasi-isometric).
	 	We chose a constant $c$ with $c^2 > ||WT_0||>0$, and consider the operator $A$ on $\h=\overline{\mathcal{R}(T)}\oplus\n(T^*)$ with the representation 
	 	\begin{equation}\label{eq35}
	 	A=\begin{pmatrix}
	 		W^*W &0 \\
	 		0 & c^2I
	 	\end{pmatrix}.
	 	\end{equation}
	 	
	 	Since $W^*W \geq I$ and $W|_{\mathcal{R}(T)}$ is an isometry, we have $W^*W \n(W^*) \subset\n(W^*)$. 
	 	This and the fact that $\mathcal{R}(T_0)\subset \n(W^*)$ give later that $W^{*2}WT_0=0$, which is later useful. 
	 	
	 	More precisely, we have (by \eqref{eq34} and \eqref{eq35})
	 	$$T^*AT=\begin{pmatrix}
	 		W^{*2}W^2 & W^{*2}WT_0 \\
	 		T^*_0W^*W^2 & T^*_0W^*WT_0
	 	\end{pmatrix}=\begin{pmatrix}
	 	W^*W &0 \\
	 	0 & T^*_0W^*WT_0
	 	\end{pmatrix} \geq T^*T,$$
	 	having in view that $W$ is an expansive quasi-isometry. Also, since $||WT_0h||^2<c^2$ for every $h \in \n(T^*)$ with $||h||=1$, we infer from \eqref{eq34} and \eqref{eq35} the representation
	 	$$A-T^*AT=\begin{pmatrix}
	 		0 & 0 \\
	 		0 & c^2I-T^*_0W^*WT_0
	 	\end{pmatrix} \geq 0.
	 	$$
	 	Thus $\n(A-T^*AT)=\overline{\mathcal{R}(T)}$, this subspace being invariant for $T$.
	 	
	 	On the other hand, denoting as usual the modulus of $W$ by $|W|=(W^*W)^{1/2}$, we obtain 
	 	$$A^{1/2}TA^{-1/2}=\begin{pmatrix}
	 		|W|W|W|^{-1} & \frac{1}{c}|W|T_0 \\
	 		0 &0 
	 	\end{pmatrix},$$
	 	where 
	 	$$ 
	 	||(A^{1/2}TA^{-1/2})|_{\n(T^*)}||=\frac{1}{c}|||W|T_0||=\frac{1}{c}||T^*_0W^*WT_0||^{1/2}<1.
	 	$$
	 	
	 	Hence the operators $T$ and $A$ satisfy the conditions of statement (ii) in Theorem \ref{thm27}, which ensures that $T$ has a natural quasi-isometric lifting and similar to an isometry.
	 	
	 	To express the condition (i) of Theorem \ref{thm27} in this case, notice that $A^{1/2}TA^{-1/2}=\widehat{T}$ is even the contraction similar to $T$ by $A^{-1}$, so $$|W|W|W|^{-1}=|W|J$$ is a contraction, where $J$ is the partial isometry from the polar decomposition of $W$.
	 	
	 	Finally, we have $A-T^*T=0 \oplus (c^2I-T^*_0T_0)$ on $\h=\overline{\mathcal{R}(T)} \oplus \n(T^*)$ and $$\mathcal{R}(A-T^*T)=\n(T^*)=\mathcal{R}(A-T^*AT),
	 	$$
	 	 because the operators $c^2I-T^*_0T_0$ and $c^2I-T^*_0W^*WT_0$ are invertible in $\mathcal{B}(\n(T^*))$.
	 \end{proof}
	
	The particular case from the last assertion of Theorem \ref{thm24} can be analyzed, in the context of this theorem. 
	
	\begin{corollary}\label{c38}
	Let $T$ and $A$ be as in Theorem \ref{thm37}, having the representations \eqref{eq34} and \eqref{eq35}, respectively. Then $\R(T)\subset \n(A-I)$ if and only if $T$ is a quasi-isometry.
	\end{corollary} 
	 
	 \begin{proof}
Assume that $\R(T)\subset \n(A-I)$, i.e., $AT=T$. By \eqref{eq34} and \eqref{eq35} this means $(W^*W-I)W=0$ and $(W^*W-I)T_0=T_0$, the former equality being always assured, because $W$ is an expansive quasi-isometry. Thus $\R(W)\oplus \R(T_0)\subset \n(W^*W-I)\subset \R(T)$. But it is easy to see from the previous proof that $\overline{\R(T_0)}=\n(W^*)$, and on the other hand, as $\n(W^*W-I)\subset \R(T)$ and $W^*W=T^*T|_{\overline{\R(T)}}$, it follows that $\overline{\R(T)}\subset \n(T^*T-I)$. In other words, $T|_{\overline{\R(T)}}=W$ is an isometry and consequently, $T$ is a quasi-isometry.

Conversely, if $T$ is quasi-isometric then $W$ is an isometry on $\overline{\R(T)}$ and $A|_{\R(T)}=I$ in \eqref{eq35}, which means that $AT=T$, in this case.	 
	 \end{proof}
	 
	 In fact, this corollary shows, by Theorem \ref{thm24}, that if $T$ is not a quasi-isometry, which satisfies the hypothesis of Theorem \ref{thm37}, then $T$ has a left invertible and natural quasi-isometric lifting $S$, which does not satisfy the range condition $S^*S\R(S)\subset \R(S)$. This happens even if $\R(A-T^*AT)=\n(T^*)=\R(A-T^*T)$ for some operator $A$, as in the proof of Theorem \ref{thm37}. 
	 
\medskip	 
	 


{\bf Statements and declarations}

{\bf Founding} The first author, Laurian Suciu, was supported by a project financed by Lucian Blaga University of Sibiu, through the research grant LBUS-IRG-2023, No. 3566/24.07.2023. 
	 
{\bf Data availability statement} Data sharing not applicable to this article as no
data sets were generated or analyzed during the current study.

{\bf Conflict of interest} The authors declare that they have no conflicts of interest relevant to the content of this article.

\newpage	 
\bibliographystyle{amsalpha}

\end{document}